\documentclass[a4paper,10pt,reqno]{amsart}
\usepackage{verbatim}
\usepackage{amssymb}

\usepackage{enumerate}
\usepackage[active]{srcltx}
\numberwithin{equation}{section}

\usepackage{t1enc}
\usepackage[utf8x]{inputenc}

\newtheorem{theorem}{Theorem}[section]
\newtheorem{lemma}[theorem]{Lemma}

\newtheorem{corollary}[theorem]{Corollary}

\theoremstyle{definition}
\newtheorem{definition}[theorem]{Definition}

\theoremstyle{remark}

\newcommand{\mc}[1]{\mathcal{#1}}

\newcommand{\setm}{\setminus}
\newcommand{\empt}{\emptyset}
\newcommand{\subs}{\subset}

\def\<{\left\langle}
\def\>{\right\rangle}
\def\br#1;#2;{\bigl[ {#1} \bigr]^ {#2} }

\newcommand{\kcC}[1]{$#1$-\cC}
\newcommand{\cC}{cellular-compact}
\newcommand{\inte}{\operatorname{int}}

\newcommand{\aaa}[1]{(a_{#1})}
\newcommand{\aaamin}[1]{(a^-_{#1})}

\author[I. Juh\'asz]{Istv\'an Juh\'asz}
\address      { Alfr\'ed Rényi Institute of Mathematics 
 }
\email{juhasz@renyi.hu}

\author[L. Soukup]{Lajos Soukup}
\thanks
  {
   }
\address
      { Alfr{\'e}d R{\'e}nyi Institute of Mathematics 
      }
\email{soukup@renyi.hu}

\author[Z. Szentmikl\'ossy]{Zolt\'an Szentmikl\'ossy}
\address{E\"otv\"os University of Budapest}
\email{szentmiklossyz@gmail.com}


\title[On cellular-compact spaces]{On cellular-compact spaces}
\thanks{The research on and preparation of this paper was
supported by  OTKA grants no. K113047 and   K129211}
\date{\today}

\begin{document}

\begin{abstract}
  As it was introduced by Tkachuk and Wilson in \cite{TW}, a topological space $X$ is {\em cellular-compact} if for any cellular,
  i.e. disjoint, family $\mathcal U$
  of non-empty open subsets  of $X$
   there is a compact  subspace
  $K\subset X$ such that $K\cap U\ne \emptyset$ for each $U\in \mc U$.

  In this note we answer several questions raised in \cite{TW} by showing that
  \begin{enumerate}[(1)]
  \item    any first countable cellular-compact  $T_2$ space is  $T_3$, and
  so its cardinality is at most  $\mathfrak{c} = 2^{\omega}$;
  \item  $cov(\mc M)>{\omega}_1$ implies
  that every first countable and separable cellular-compact $T_2$ space is compact;
  \item if there is no $S$-space then any cellular-compact
  $T_3$ space of countable spread is compact;
  \item $MA_{{\omega}_1}$ implies that every point of a compact $T_2$ space of countable spread has a disjoint local $\pi$-base.
   \end{enumerate}

\end{abstract}

\maketitle

\section{Introduction}
A topological space $X$ is said to be {\em \kcC {\kappa}}
  if for any cellular family $\mc U$ of open subsets of $X$
  with $|\mc U|={\kappa}$ there is a compact 
  $K\subs X$ such that $K\cap U\ne \empt$ for each $U\in \mc U$.
  $X$ is {\em \cC} iff it is \kcC{{\kappa}}
  for all  cardinals ${\kappa}$.

In  \cite[Theorem 4.13]{TW} the authors proved that the cardinality of a first countable cellular-compact
$T_3$ space does not exceed the cardinality of the continuum and asked the natural question if
this result can be extended to $T_2$ spaces:

\noindent {\sc Question 5.1} {\em Let $X$ be a \cC\ first countable
$T_2$ space. Is it true that $|X|\le 2^{\omega}$?}

In \cite{B}   a partial answer was given to this question by showing that this statement 
can be extended to the class of Urysohn spaces.
In Theorem \ref{tm:ckt2t3} we give the full affirmative answer by proving that, somewhat surprisingly,
any first countable and \cC\ $T_2$  space is actually $T_3$.

In \cite{TW}, under CH, a non-compact Tychonov \cC\ space was constructed which is both
first countable and separable. Consequently,
the authors raised the following question:

\noindent {\sc Question 5.2} {\em Does there exist a model of ZFC in
 which every Tychonov \cC\ space that is both separable and
first countable is compact? }

Our Theorem \ref{tm:covMw1-cC2compact}
gives an affirmative answer to this question by showing that the assumption $cov(\mc M)>{\omega}_1$,
or equivalently $MA_{\omega_1}$(countable), implies  that every  \cC\
separable and first countable $T_2$ space is compact.

The following two closely related questions were also raised in \cite{TW}.

\noindent {\sc Question 5.3} {\em Let $X$ be a hereditarily separable
cellular-compact space.
Must $X$ be compact?}

\noindent {\sc Question 5.4} {\em Let $X$ be a \cC\ space of countable spread.
Must $X$ be compact?}

In Theorem \ref{tm:cKhd} below we show that the answer to both questions is  "yes" 
for $T_3$ spaces, provided that there are no S-spaces, e.g. under PFA.

In the closely related Corollary \ref{co:cKhdlk} we give a consistent affirmative answer to
{\sc Question 5.8} of \cite{TW} by proving that, under $MA_{{\omega}_1}$, every point
of a compact $T_2$ space of countable spread has a disjoint local $\pi$-base.

Our notation is standard. For any topological space $X$ we use $\tau(X)$ to denote the
family of all open sets in $X$. For a subset $A \subs X$ we use $A'$ to denote the
derived set of all accumulation points of $A$.
Concerning cardinal functions we follow the notation of \cite{Ju}.

\section{When $T_2$ implies $T_3$ for \cC\ spaces}\label{23}

Let us recall that a subset $Y\subs X$ is called {\em strongly discrete (SD)} iff there is
    a disjoint neighborhood assignment $u:Y\to \tau(X)$. Equivalently, SD sets are the
    ranges of choice functions on cellular families. While every infinite set in a $T_2$ space
    has an infinite discrete subset, there are $T_2$ spaces in which some infinite sets have
    no infinite SD subset, see e.g. the example in § 4 of \cite{HJ}. Below we introduce
    a property of subsets that will allow us to circumvent this inconvenient phenomenon
    in general $T_2$ spaces.

     \begin{definition} Let $X$ be any topological space.
      A subset $Y\subs X$ is called {\em fluffy} (in $X$)  iff there is a cellular family
      $\{U_y:y\in Y\}\subs {\tau}(X)$   such that
      $y\in \overline{U_y}$ for each $y\in Y$.
      \end{definition}

Clearly, every SD set is fluffy. Now, it turns out that every infinite set in any $T_2$ space
has an infinite fluffy subset, but to prove that we need the following simple auxiliary result.

\begin{lemma}\label{lm:a-or-b}
    Assume that $U$ is an open set in a topological space $X$ and the point
    $p\notin \inte{\overline{U}}$. If and $A$ and $B$ are disjoint open sets in $X$ then
      \begin{displaymath}
        p\notin \inte{(\overline{U\cup A})}\cap \inte{(\overline{U\cup B})}.
      \end{displaymath}
    \end{lemma}

    \begin{proof}
      Assume, on the contrary,
      that $p\in V=\inte{\overline{U\cup A}}\cap\inte{\overline{U\cup B}}$.
      Then $p\notin \inte{\overline{U}}$ implies $p \in W = V \setminus \overline{U} \ne \emptyset$,
      while $W \cap \overline{U} = \emptyset$ implies that
      $W\subs \inte{\overline A}\cap \inte{\overline B}$. This, however, contradicts that
      $\inte{\overline A}\cap \inte{\overline B} = \empt$ because $A$ and $B$ are disjoint open sets.
      \end{proof}

      \begin{theorem}\label{tm:fluffy}
      Every infinite set $Y$ in a $T_2$  space $X$
      has an infinite  fluffy subset.
      \end{theorem}

\begin{proof}
      By induction on $n < \omega$ we shall define open sets $U_n$ and points $y_n\in Y \cap \overline {U_n}$
      such that the family $\{U_n:n\in {\omega}\}$ is cellular
      and for each $n$ the the following inductive hypothesis $*(n)$ holds:
      \begin{displaymath}\tag{$*(n)$}
      \text{The set } Y_n = Y\setm \inte{\big (\overline {\bigcup _{k<n}U_k}\big)} \text{ is infinite}.
      \end{displaymath}
      Note that we trivially have $Y_0 = Y$, hence $*(0)$ holds.

      Now, assume that we have  defined $U_k$ and $y_k$ for all $k<n$ satisfying $*(n)$. Pick distinct
      points $a$ and $b$ from $Y_n \setminus \{y_k : k < n\}$ and fix an open neighborhood $A$ of $a$ and  an open neighborhood $B$ of $b$
      with $A \cap B = \emptyset$. Then we may apply Lemma \ref{lm:a-or-b} with $U=\bigcup _{k<n}{U_k}$ for all $p \in Y_n$
      to obtain that $$Y_n = [Y_n \setminus \inte{\overline {U\cup A}}] \cup [Y_n \setminus \inte{\overline {U\cup B}}].$$
      By symmetry, we may thus assume that $Y_n \setminus \inte{\overline {U\cup A}}$ is infinite.

      We then put $y_n = a$ and $U_n = A \setminus \overline {\bigcup _{k<n}U_k}$. We have to show that $y_n = a \in \overline {U_n}$.
      If $a \notin \overline {\bigcup _{k<n}U_k}$ then we even have $a \in U_n$.
      If $a \in \overline {\bigcup _{k<n}U_k}$ then $a$ is a boundary point of $\overline {\bigcup _{k<n}U_k}$ and hence of $X \setminus \overline {\bigcup _{k<n}U_k}$ as well.
      But then $a$ is a boundary point of  $A \setminus \overline {\bigcup _{k<n}U_k}$ as well because $A$ is a neighborhood of $a$.
      It is also obvious from this definition that $$Y_{n+1} = Y_n \setm \inte{\big (\overline {\bigcup _{k<n}U_k \cup A}\big)},$$
      is infinite, hence $*(n+1)$ is satisfied as well. This completes our induction and yields the infinite fluffy
      subset $\{y_n:n\in {\omega}\}$ of $Y$.
\end{proof}

We shall use fluffy sets to show that first countable \cC\ $T_2$  spaces are $T_3$,
as was promised in the introduction. Actually, something weaker than first countability
will suffice for that, namely having {\em countable closed pseudocharacter}. Let us recall, see e.g. \cite{Ju},
that the closed pseudocharacter $\psi_c(p,X)$ of a point $p$ in a space $X$ is defined by

        \begin{displaymath}
        \psi_c(p,X)=\min\{|\mc U|: \mc U\subs \tau(X) \text{ and }
        \{x\}=\bigcap_{U\in \mc U} {U}=\bigcap_{U\in \mc U} \overline{U}\}.
        \end{displaymath}
        Of course, if $X$ is $T_2$ then $\psi_c(p,X)$ is defined for all $p \in X$
        and then so is
        \begin{displaymath}
        \psi_c(X)=\sup\{\psi_c(p,X):p\in X\},
        \end{displaymath}
        the closed pseudocharacter of $X$. Note that $p$ is an isolated point in $X$ iff $\psi_c(p,X) = 1$.
        Also, it is obvious that $\psi_c(p,X) \le \chi(p,X)$ for all $p \in X$, hence $\psi_c(X) \le \chi(X)$.
        Finally, if $\psi_c(p,X)={\omega}$ then there is a sequence $\{V_n:n\in {\omega}\}$ of open neighborhoods of $p$
        witnessing this which is also decreasing.

    \begin{lemma}\label{lm:almostconvergeinG}
    Assume that $X$ is a $T_2$ space and $p \in X$ with $\psi_c(p,X)={\omega}$, moreover
    $\{V_n:n\in {\omega}\}$ is a decreasing sequence of open neighborhoods of $p$
    witnessing this.\\
    \noindent (i) If $a_n \in \overline{V_n} \setm \{p\}$ for all $n\in {\omega}$
    then $\{a_n:n\in {\omega}\}'\subs \{p\}$.\\
    \noindent (ii) If $A \subs G\in {\tau}(X)$ and $p\in A'$ then there is a cellular family
    $\{U_n:n\in {\omega}\}$ of open sets such that $U_n \subs G \cap V_n$ and
    $A\cap U_n\ne \empt$ for each $n\in {\omega}$.
    \end{lemma}

    \begin{proof}
    (i) Assume that $\{a_n:n\in {\omega}\}$ is given with $a_n \in \overline{V_n} \setm \{p\}$ for all $n\in {\omega}$.
    Then $\{a_m:m \ge n\}\subs \overline{V_n}$ and so $\{a_n:n\in {\omega}\}' = \{a_m:m \ge n\}'\subs \overline{V_n}$ for all $n$,
hence $$\{a_n:n\in {\omega}\}'\subs \bigcap_{n\in {\omega}}\overline{V_n} = \{p\}.$$
(ii) The required sequence $\{U_n:n\in \}$ together with points $a_n \in A \cap U_n$  will be constructed by induction on $n$ in
    such a way that $U_n\subs V_n\cap G$ and $p \notin \overline {U_n}$ for all $n < \omega$.
Assume that we have constructed $\{U_k:k<n\}$, then
$p \in W_n= V_n\setm \overline{\bigcup_{k<n}U_k}$, hence
    $p\in (A\cap W_n)'$. We may then pick a point $a_n\in (A\cap W_n)\setm \{p\}$ and an open neighborhood
    $U_n$ of $a_n$ such that $U_n\subs G\cap W_n$ and $p\notin \overline {U_n}$.
    So the inductive step can be carried out and the family $\{U_n:n\in {\omega}\}$
    with the required properties is constructed.
\end{proof}

\begin{definition}
    The space $X$ is said to have property $\aaa{\kappa}$ (resp. $\aaamin{{\kappa}}$)
    iff every SD-subset of $X$ of size $ {\kappa}$ has compact closure (resp. has a complete accumulation point).
    \end{definition}

\begin{lemma}\label{lm:fluffycomp}
    Assume that $X$ is a $T_2$ space with property $\aaamin{\omega}$ and $Y \subs X$ is fluffy such that
    $\psi_c(y,X) = {\omega}$ for each $y\in Y$. Then there is an SD subset $D$ of $X$ with $|D| = |Y|$
    for which $Y \subs \overline{D}$. Moreover, if $X$ is also $|Y|$-\cC\ then $D$ can be chosen so that
    $\overline{D}$, and hence $\overline{Y}$ as well, is compact.
    \end{lemma}

\begin{proof}
Since any finite subset of a $T_2$ space is SD, we may assume that $Y$
is infinite. Now, fix a cellular family $\{U_y : y \in Y\} \subs {\tau}(X)$ witnessing that  $Y$ is fluffy.
By Lemma \ref{lm:almostconvergeinG}, for each $y\in Y$ there is a cellular family
    $\mc U_y = \{U_{y,n} : n < \omega\}$ in $U_y$ such that if $a_{y,n} \in U_{y,n}$ for each $n < \omega$
then $\{a_{y,n}:n\in {\omega}\}'\subs \{y\}$. But now we know that $\{a_{y,n}:n\in {\omega}\}' \ne \emptyset$,
hence $y \in \{a_{y,n}:n\in {\omega}\}'$. Clearly, then $D = \{a_{y,n} : y \in Y,\,n\in {\omega}\}$ is as required.

If, in addition, $X$ is $|Y|$-\cC\ then there is a compact subset $K$ of $X$ such that $K \cap U_{y,n} \ne \emptyset$
for all $y \in Y$ and $n \in \omega$, hence all the points $a_{y,n}$ may be chosen from $K$, and then
$\overline{D} \subs K$.
    \end{proof}

We are now ready to present the main result of this section.

 \begin{theorem}\label{tm:ckt2t3}
    Let $X$ be a $T_2$ space of countable closed pseudocharacter and assume
    that the closure of every countable SD
    subset of $X$ is countably compact. Then
    $X$ is countably compact, $T_3$, and first countable.
    \end{theorem}

\begin{proof}
By Theorem \ref{tm:fluffy} any infinite subset of $X$ has a countably infinite fluffy subset,
hence the countable compactness of $X$ follows if we show that every countably infinite fluffy
set $Y \subs X$ has an accumulation point in $X$. If $Y$ contains infinitely many isolated
points then this is immediate from our assumption because any set of  isolated
points is SD. So, we may assume that $\psi_c(y,X) = {\omega}$ for each $y\in Y$.
But our assumption clearly implies that $X$ possesses property $\aaamin{\omega}$,
hence in this case we may apply Lemma \ref{lm:fluffycomp} to conclude that even
$\overline{Y}$ is countably compact.

Next, to show that $X$ is $T_3$, consider a closed set $F \subs X$ and a point $p \notin F$.
Of course, we may assume that $p$ is not isolated. Then we have $\psi_c(p,X) = \omega$ and we have
a strictly decreasing sequence $\{V_n:n\in {\omega}\}$  of open neighborhoods of $p$ such that
$\bigcap_{n\in {\omega}}\overline{V_n} = \{p\}.$ We claim that there is an $n \in \omega$ for
which $F \cap \overline{V_n} = \emptyset$, hence $F$ and $p$ do have disjoint neighborhoods.

Assume, on the contrary that  $F \cap \overline{V_n} \ne \emptyset$ for each $n$.
Then, we may clearly pick distinct points $x_n \in F \cap \overline{V_n}$ and
by the countable compactness of $X$ the infinite set $\{x_n : n \in \omega\}$ has
an accumulation point $x \in F$. This, however contradicts part (i) of Lemma
\ref{lm:almostconvergeinG}.

Finally, it is well-known that in a countably compact $T_3$ space every point of
countable pseudocharacter actually has countable character, hence $X$ is indeed first countable.
\end{proof}

Clearly, if $X$ has property $\aaa {\kappa}$ then $X$ is \kcC {\kappa}.
If, on the other hand, $X$ is $T_2$ and $\psi_c(X) \le {\omega}$ then the reverse of this implication also holds.
In fact, by Lemma \ref{lm:fluffycomp} then \kcC {\kappa} implies that every fluffy subset of $X$ of size $ {\kappa}$ has compact closure.
In particular, if $X$ is $\omega$-\cC\ then it satisfies all the assumptions of Theorem \ref{tm:ckt2t3}.

\section{When \cC\ implies compact}

In \cite{TW} the authors raised several questions if \cC\ spaces with certain properties
are necessarily compact. For instance, they constructed with the help of CH a
first countable and separable \cC\ Tychonov space that is not compact and asked if
this can be done in ZFC. Our next result gives a negative answer to this question.

\begin{theorem}\label{tm:covMw1-cC2compact}
If $cov(\mc M)>{\omega}_1$ then every first countable and separable
\cC\  $T_2$ space is compact.
\end{theorem}

We actually start with giving a preparatory result. We recall that a $T_2$
space $X$ is called {\em $\pi$-regular} if every non-empty open
set in $X$ includes a non-empty regular closed set. This property is clearly weaker than regularity.

We also recall that $cov(\mc M)>\mu$, the statement that the covering number of the meager ideal $\mc M$
on the reals is $> \mu$, is equivalent with $MA_\mu$(countable), as we shall use it in this form.

\begin{lemma}\label{lm:covMw1-disjointPibase}
  Assume that $X$ is a  $\pi$-regular space of countable $\pi$-weight
with $\mc A$ a fixed  (countable) $\pi$-base of $X$,  moreover
$H$ is a nowhere dense subset of $X$ such that $|H| < cov(\mc M)$
and $$\sup\{\chi(h,X):h\in H\}< cov(\mc M).$$ Then $\mc A$ has a {\em disjoint} subfamily
$\mc B$ that is simultaneously a local $\pi$-base for all points $h\in H$.
Consequently, if $X$ is also \cC\ then $\overline{H}$ is compact.
\end{lemma}

\begin{proof}

 Let us consider the countable partial order $\mc P=\<P,\supseteq\>$, where
 \begin{displaymath}
 P=\{p\in \br \mc A;<{\omega};: p  \text{ is disjoint and } \overline{\cup p}\cap \overline H=\empt\}.
 \end{displaymath}

For any open set $U$ we let
\begin{displaymath}
 D_{U}=\{p\in P: \exists V\in p\ V\subs U\}.
\end{displaymath}

\noindent{\bf Claim.} For every open set $U$ with $U \cap H \ne \emptyset$ the family {\em $D_U$ is dense in $\mc P$.}

Indeed, consider an arbitrary  $p\in P$. Then $U \cap H \ne \emptyset$ implies that the open set  $W_0=U\setm \overline {\cup p} \ne \emptyset$.
Since $H$ is nowhere dense, $W_1=W_0\setm \overline H\ne \empt$ as well, hence
there is a non-empty open $S$ such that $\overline S\subs W_1$ because $X$ is $\pi$-regular.
Now, if $V\in \mc A$ is such that $V\subs S$ then $q=p\cup\{V\}\in D_U$ and  $q \supset p$,
and the claim is verified

For each $h\in H$ let $\mc U(h)$ be an open neighborhood base of $h$ in $X$ of size ${\chi}(h,X)$
and set $\mathcal{U} = \bigcup \{U(h) : h \in H\}.$ Our assumptions then imply $|\mathcal{U}| < cov(\mc M)$,
while the Claim implies that $D_{U}$ is dense in $\mc P$ for each $U \in \mc U$. But then there is a
filter $G$ in $\mc P$ that intersects the dense set $D_{U}$ for each $U \in \mc U$. Clearly, then
$\mc B = \bigcup\,G$ is as required.

Also, if a set $S$ intersects all members of $\mc B$ then obviously we have $H \subs \overline{S}$,
in particular, $\overline{H}$ is compact if $\overline{S}$ is.
\end{proof}

\begin{proof}[Proof of Theorem \ref{tm:covMw1-cC2compact}]
First we shall prove the result with the additional assumption that $X$ is crowded,
i.e. has no isolated points.
Let us next observe that $X$ is actually $T_3$ by Theorem \ref{tm:ckt2t3},
moreover $X$ has countable $\pi$-weight being separable and first countable,
so we can apply Lemma \ref{lm:covMw1-disjointPibase} to $X$.

To see that $X$ is compact, we shall use the fact, see e.g. Lemma 1.7 in \cite{JS},
that if every free sequence in $X$ has compact closure then $X$ is compact. Now, as $X$ is crowded,
every discrete subspace, in particular every free sequence is nowhere dense in $X$.
But Lemma \ref{lm:covMw1-disjointPibase} now implies that every nowhere dense subset
of $X$ of cardinality $\le\, \omega_1$ has compact closure, while in a first countable compact space
every free sequence is countable. These observations together imply that every free sequence in $X$ is countable.
But then we can conclude that every free sequence in $X$ has compact closure and hence $X$ is compact.

In the general case, in which $X$ may have isolated points, let $I$ denote the set of all
isolated points of $X$. Then we know that $\overline{I}$ is compact because $X$ is \cC.
But then $Y = \overline{X \setm \overline{I}}$ is a regular closed set in $X$ which clearly
inherits from $X$ separability, first countability, and the \cC\ property, moreover $Y$ has
no isolated points, hence $Y$ is compact, and so $X$ is the union of two compact subsets.
\end{proof}

\smallskip

Question 5.3 (resp. 5.4) of \cite{TW} asks if a \cC\ space that is hereditarily separable (resp. has countable spread)
is actually compact. Of course, the affirmative answer to 5.4 implies the same for 5.3. The authors of \cite{TW}
failed to raise the same type of question for the closely related class of hereditarily Lindelöf ($T_2$) spaces.
Our next goal is to show that in this class \cC\ does imply compact. Interestingly, this then implies that the
affirmative answer to 5.4 is consistent, at least for $T_3$ spaces.

\begin{theorem}\label{tm:hL}
Every hereditarily Lindelöf and \cC\ $T_2$ space is compact.
\end{theorem}

\begin{proof}
According to 2.10 of \cite{Ju}, any hereditarily Lindelöf $T_2$ space has countable closed pseudocharacter,
hence if it is \cC\ then it is countably compact by our Theorem \ref{tm:ckt2t3}. But countably compact Lindelöf
spaces are compact.
\end{proof}

It is well-known that if a space $X$ of countable spread is not hereditarily Lindelöf then it has a hereditarily separable
subspace $Y$ that is not hereditarily Lindelöf. So, if $X$ is $T_3$ then $Y$ is an S-space.

\begin{corollary}\label{tm:cKhd}
Any non-compact but \cC\ $T_3$ space of countable spread has an S-subspace. 
So, if there is no S-space then every \cC\ $T_3$ space of countable spread is compact.   
\end{corollary}

Of course, it is known that there are models of ZFC in which no S-space exits, for instance
PFA implies this. The final result of this section shows that for locally compact $T_2$ spaces
already $MA_{{\omega}_1}$ suffices to get the same result. Note that it is consistent with
$MA_{{\omega}_1}$ that S-spaces exist, see e.g. \cite{Zoli}.   

\begin{corollary}\label{co:cKhdlk}
$MA_{{\omega}_1}$ implies that every locally compact and \cC\ $T_2$ space $X$ of countable spread is compact.
\end{corollary}

\begin{proof}
Assume, on the contrary, that $X$ is not compact, then by the previous corollary $X$ has an S-subspace.
let $K$ be the one point compactification of $X$. Then clearly $K$ also has countable spread,
and consequently has countable tightness. But $MA_{{\omega}_1}$ implies that a compact space of countably tightness 
has no S-subspace by a theorem of Szentmiklóssy \cite{Zoli}. We have reached a contradiction.
\end{proof}

It is an immediate consequence of this result and of Theorem 3.13 of \cite{TW} that, under $MA_{{\omega}_1}$, every point
of a compact $T_2$ space of countable spread has a disjoint local $\pi$-base. Thus $MA_{{\omega}_1}$ implies a consistent affirmative answer
to Question 5.8 of \cite{TW} that asked this for hereditarily separable compacta.

\end{document}